\documentclass[11pt]{article}

\usepackage{fullpage} 
\usepackage{authblk}
\usepackage{ucs}
\usepackage{amssymb}
\usepackage{amsthm}
\usepackage{amsmath}
\usepackage{latexsym}
\usepackage[cp1251]{inputenc}
\usepackage[english]{babel}
\usepackage{graphicx}
\usepackage{wrapfig}
\usepackage{caption}
\usepackage{subcaption}
\usepackage{txfonts}
\usepackage{lmodern}
\usepackage{mathrsfs}
\usepackage{algorithm}
\usepackage{dsfont}
\usepackage{enumerate}
\usepackage{hyperref}
\usepackage{multicol}
\usepackage{csquotes}
\usepackage{stmaryrd}
\usepackage{tikz}
\usepackage{comment}
\usepackage{enumitem}
\usepackage[initials]{amsrefs} 
\usepackage{cite}

\newcommand{\badrounding}[1]{}
\newcommand{\cC}{\mathcal{C}}

\newtheorem{theo}{Theorem}
\newtheorem{prop}[theo]{Proposition}
\newtheorem{lemma}[theo]{Lemma}

\newtheorem{conj}[theo]{Conjecture}

\theoremstyle{definition}

\numberwithin{theo}{section}

\tikzset{
vtx/.style={inner sep=1.1pt, outer sep=0pt, circle, fill,draw}, 
vtxl/.style={inner sep=1.1pt, outer sep=0pt, rectangle, fill=yellow,draw=black}, 
hyperedge/.style={fill=blue,opacity=0.3,draw=black}, 
}

\begin{document}

\author[1]{Maria Axenovich \footnote{\texttt{maria.aksenovich@kit.edu} Research is partially supported by DFG grant FKZ AX 93/2-1}}
\author
[2]{J\'ozsef Balogh \footnote{\texttt{jobal@illinois.edu}. Research is partially supported by NSF grant DMS-1764123, NSF RTG grant DMS 1937241, Arnold O. Beckman Research Award (UIUC Campus Research Board RB 22000), the Langan Scholar Fund (UIUC).}}
\author[1]{Felix Christian Clemen \footnote{\texttt{felix.clemen@kit.edu}}}
\author[1]{Lea Weber \footnote{\texttt{lea.weber@kit.edu}}}
\affil[1]{Karlsruhe Institute of Technology, 76133 Karlsruhe, Germany}
\affil[2]{University of Illinois at Urbana-Champaign, Urbana, Illinois 61801, USA}

\title{Unavoidable order-size pairs in hypergraphs --  positive forcing density}

\maketitle 
\begin{abstract}
  Erd\H{o}s, F\"uredi, Rothschild and S\'os initiated  a study of  classes of graphs that forbid every induced subgraph on a given number $m$ of vertices and number $f$ of edges.  
    Extending their notation to $r$-graphs, we write $(n,e) \to_r (m,f)$ if every $r$-graph $G$ on $n$ vertices with $e$ edges has an induced subgraph on $m$ vertices and $f$ edges.
      The \emph{forcing density} of a pair $(m,f)$ is $$ \sigma_r(m,f) =\left. \limsup\limits_{n \to \infty}  \frac{|\{e : (n,e) \to_r (m,f)\}|}{\binom{n}{r}} \right. .$$
  In the graph setting it is known that there are infinitely many pairs $(m, f)$ with positive forcing density. Weber asked if there is a pair of positive forcing density for $r\geq 3$ apart from the trivial ones $(m, 0)$ and $(m, \binom{m}{r})$.
  Answering her question, we show that $(6,10)$ is such a pair for $r=3$ and conjecture that it is the unique such pair. Further, we find necessary conditions for a pair to have positive forcing density, supporting this conjecture. 
\end{abstract}

\section{Introduction}
  The \emph{Tur\'an function} $\textup{ex}(n,H)$ is the maximum number of edges in an $H$-free $n$-vertex $r$-graph. The \emph{Tur\'an density} of $H$, denoted by $\pi(H)$, is defined as follows
 \[
 \pi(H)=\lim_{n\to \infty} \frac{\textup{ex}(n,H)}{ \binom{n}{r}}.
 \]
Determining the Tur\'an function for graphs and hypergraphs is a central topic in extremal graph theory with many challenging open problems, trying to identify what graph density forces the occurrence of a specific subgraph. Here, we are concerned with conditions on the graph density that forces the occurrence of an induced subgraph on a given number of vertices  and a given number of edges, i.e., a given order-size pair.   Erd\H{o}s, F\"uredi, Rothschild and S\'os~\cite{EFRS} studied the class of graphs that does not contain a vertex subset of a given size $m$ that spans exactly $f$ edges. Given pairs of non-negative integers $(n,e)$ and $(m,f)$ we write $$(n,e)\rightarrow_r(m,f)$$ if every $r$-graph $G$ on $n$ vertices and with $e$ edges contains a vertex subset of a given size $m$ that spans exactly $f$ edges. The \emph{forcing density} of a pair $(m,f)$ is
\begin{align*}
    \sigma_r(m,f)=\limsup_{n \to \infty} \frac{|\{ e: (n,e) \rightarrow_r (m,f)\}|}{\binom{n}{r}}.
\end{align*}

Erd\H{o}s, F\"uredi, Rothschild and S\'os   \cite{EFRS} studied $\sigma_2(m,f)$ for different choices of $(m,f)$. They showed that 
if $(m,f) \in \{(2,0), (2,1), (4,3), (5,4), (5,6)\}$, then $\sigma_2(m,f) =1 $; otherwise, $\sigma_2(m,f) \le \frac23$. They also gave a construction that shows that for  most pairs $(m,f)$ we have $\sigma_2(m,f) = 0$.
The upper bound $\frac 23$  was subsequently improved by He, Ma, and Zhao~\cite{HMZ} to $ \frac12$. On the other hand, Erd\H{o}s, F\"uredi, Rothschild and S\'os   \cite{EFRS} showed that there are infinitely many pairs of positive forcing density, in particular there are infinitely many pairs $(m,f)$ with $\sigma_2(m,f) \ge  \frac 18$.  He, Ma, and Zhao~\cite{HMZ} improved this result, by showing that there are infintely many pairs $(m,f)$ with $\sigma_2(m,f) \ge  \frac 12$.
Considering the hypergraph setting, Weber \cite{W22} showed that for any $r, m \in \mathbb N$, $r, m \ge 3$, all but at most $m^{\frac{r}{r-1}}$ of all possible $\binom mr$ pairs $(m,f)$ satisfy $\sigma_r(m,f) = 0$.\\

Axenovich and Weber \cite{AW} asked whether there are pairs $(m,f)$ for which not only $\sigma_r(m,f)=0$, but a stronger statement holds.
A pair $(m,f)$ is {\it absolutely r-avoidable} if there is $n_0$ such that for each $n>n_0$ and for every $e\in \{0, \ldots, \binom{n}{r}\}$, $(n,e) \not\to_r  (m,f)$. 
In \cite{AW} it was shown that for $r=2$ there are infinitely many absolutely avoidable pairs. Moreover, there is an infinite family of absolutely avoidable pairs of the form $(m, \binom m2/2)$ and  for every sufficiently large $m$, there exists an $f$ such that $(m,f)$ is absolutely avoidable. In \cite{W22} this result was extended to higher uniformities to show that for every
 $r \ge 3$, there exists $m_0$ such that for every $m \ge m_0$ either $(m, \lfloor\binom mr/2\rfloor)$ or $(m, \lfloor{\binom mr/2}\rfloor -m-1)$ is absolutely avoidable. \\

While there are many pairs $(m,f)$ for which $\sigma_r(m,f)=0$,  not a single (non-trivial) pair with positive forcing density was known for $r$-graphs when $r\geq 3$. We denote by $K_t^{r}$ the $r$-graph on $t$ vertices where every $r$-set is an edge. Note that $\sigma_r(r, 1) = \sigma_r(r,0) = 1$ and for $f = 0$,  $\sigma_r$ corresponds to the Tur\'an density, i.e., 
$ \sigma_r(m,0) = \sigma_r(m,\tbinom mr) = \pi(K_{m}^{r}) $,  where the best currently known general bounds on the Tur\'an density are 
$$ 1- \left(\frac{r-1}{m-1}\right)^{r-1}\le \pi(K_{m}^{r}) \le 1 - \binom{m-1}{r-1}^{-1},$$
due to Sidorenko~\cite{S81} and de Caen~\cite{dC}. 
Weber \cite{W22}  asked whether for $m>r\geq 3$, there is any $f$ with $0<f<\binom{n}{r}$ such that $\sigma_r(m,f)>0$ and suggested the pair $(6,10)$ as a candidate. We answer this question in the affirmative and prove $\sigma_3(6,10)>0$. \\

Given families of $r$-graphs $\mathcal{F},\mathcal{G}$, we denote by $\textup{ex}(n,{}_{\mathrm{ind}}\mathcal{F},\mathcal{G})$ the maximum number of edges in an $n$-vertex $r$-graph not containing any $F\in\mathcal{F}$ as an induced copy and also not any $G\in\mathcal{G}$ as a copy. Further, denote by $\pi({}_{\mathrm{ind}}\mathcal{F},\mathcal{G})$ the limit
 \[
 \pi({}_{\mathrm{ind}}\mathcal{F},\mathcal{G})=\limsup_{n\to \infty} \frac{\textup{ex}(n,{}_{\mathrm{ind}}\mathcal{F},\mathcal{G})}{ \binom{n}{r}}.
 \]  
We mostly consider $3$-graphs in this paper. When clear from context, we shall write $abc$ for the set $\{a,b,c\}$ corresponding to an edge in a $3$-graph. Denote by $[n]=\{1,2,\ldots,n\}$ the set of the first $n$ integers. The $3$-graph on vertex set $[4]$ with edgeset $\{123,124,124\}$ is denoted by $K_{4}^{3-}$. Let $\mathcal{F}_6^{10}$ be the family of $6$-vertex $3$-graphs containing exactly $10$ edges.

\begin{theo}
\label{610theo}
We have that  $\sigma_3(6,10)= 1-2\pi({}_{\mathrm{ind}}\mathcal{F}_6^{10},\{K_4^{3-}\})$. Moreover,  $0.42622\leq \sigma_3(6,10)\leq  0.47106$.
%
\end{theo}

We do not know whether other pairs $(m,f)$ with $m>3$, $0 < f < \binom m3$  exist, such that $\sigma_3(m,f)>0$. It seems plausible that for $r=3$ there are indeed no other pairs with positive forcing density.  

\begin{conj}
\label{conj610}
Let $m$ and $f$ be positive integers, $0<f<\binom{m}{3}$. If $\sigma_3(m,f)>0$, then $(m,f)=(6,10)$.
 \end{conj}
 The following result provides evidence for this conjecture to be true. 

 \begin{theo}\label{Diophantine}
Let $m$ and $f$ be positive integers,  $0 < f < \binom m3$. If $\sigma_3(m,f) > 0$, then 
there exist $x_1, x_2, x_3 \in [m-1]$ such that
\begin{align}
\label{dioeq} f = \binom{x_1}3 = \binom m3 - \binom{x_2}3 = \binom{x_3}3 + \binom{x_3}2(m-x_3).
\end{align}
\end{theo}

Thus, in particular if there are no other non-trivial solutions except for $m=6$, $x_1= 5, x_2= 5, $  $x_3=3$, to the above Diophantine equation, then Conjecture~\ref{conj610} is true.  A computer search for suitable solutions of \eqref{dioeq} did not give a result for $m\leq 10^6$.

This paper is organized as follows: In Section~\ref{sec:Thm1} we prove Theorem~\ref{610theo}. In Section~ \ref{sec:Thm2} we prove Theorem~\ref{Diophantine}. Finally, in Section~\ref{sec:conclu} we make concluding remarks and state open problems.

\section{Proof of Theorem~\ref{610theo}}
\label{sec:Thm1}

We say a $3$-graph $G$ \emph{induces} $(6,10)$ if $G$ contains an induced copy of some $F\in \mathcal{F}_6^{10}$. If $G$ does not contain any $F\in \mathcal{F}_6^{10}$ as an induced copy, we say $G$ is $(6,10)$-free, i.e., a $3$-graph is $(6,10)$-free if no $6$-vertex set induces exactly $10$ edges.

\subsection{Proof idea}
Before proving Theorem~\ref{610theo}, we give a short sketch of the proof. 
We shall show that for every $\epsilon>0$  there is $n_0$ such that for every $n>n_0$ 
if $G$ is an $n$-vertex $3$-graph satisfying
\begin{align}
\label{edge density}
    \frac{e(G)}{\binom{n}{3}}\in \left[\pi({}_{\mathrm{ind}}\mathcal{F}_6^{10},\{K_4^{3-}\})+\varepsilon , 1-\pi({}_{\mathrm{ind}}\mathcal{F}_6^{10},\{K_4^{3-}\})-\varepsilon  \right],
\end{align}
then $G$ induces $(6,10)$.
Then we first use a standard Ramsey type argument to partition most of the vertices of $G$ into many large homogeneous sets. First, we rule out the case that there is a large clique and a large independent set that are disjoint. Thus, most of the vertex set of $G$ or its complement $G^c$ can be partitioned into large independent sets. Due to the symmetry of the problem, if we find a $(6,10)$-set in $G^c$, we also find a $(6,10)$-set in $G$. Thus, without loss of generality, we can assume that most of the vertices of $G$ can be partitioned into many large independent sets. Using a classical supersaturation result and the density assumption on $G$, we find many copies of $K_4^{3-}$ in $G$ and thus, in particular, four large independent sets spanning many transversal copies of $K_4^{3-}$. Using a final cleaning argument, we find a $(6,10)$-set in this substructure.  

On the other hand, we fix an arbitrary $3$-graph $G$ on ${\mathrm{ex}}(n, {}_\mathrm{ind}\mathcal{F}_6^{10},\{K_4^{3-}\})$ 
edges that is $(6,10)$-free and $K_4^{3-}$-free. Then every set of $6$ vertices spans at most $9$ edges, so there is a graph on $n$ vertices and $e$ edges, for any  $e\leq {\mathrm{ex}}(n, {}_{\mathrm{ind}}\mathcal{F}_6^{10},\{K_4^{3-}\})$, that is $(6,10)$-free. By taking complements, there also is a graph on $n$ vertices and $e$ edges for every $e \geq \binom{n}{3} -{\mathrm{ex}}(n, {}_{\mathrm{ind}}\mathcal{F}_6^{10},\{K_4^{3-}\})$, that is $(6,10)$-free. 

\subsection{Definitions, notations, and construction}

An \emph{independent set} in an $r$-graph is a vertex subset containing no edges. A \emph{clique} in an $r$-graph is a vertex subset in which every $r$-set is an edge. A \emph{homogeneous set} in an $r$-graph is a clique or an independent set. \\

Let $G$ be a $3$-graph and let $X, Y, Z \subseteq V(G)$, not necessarily disjoint from each other. Then, let $E_G(X, Y, Z) = \{(x, y, z) \in E(G) : x \in X, y \in Y, z \in Z, x, y, z $ pairwise distinct$\}$. We say $E_G(X,Y,Z)$ is \emph{complete} if $E_G(X,Y,Z) = \{(x,y,z): x \in X, y \in Y, z \in Z, x, y, z $ pairwise distinct$\}$, and $E_G(X,Y,Z)$ is \emph{empty} if $E_G(X,Y,Z) = \emptyset$. If the $3$-graph $G$ is clear from the context, we might omit the index and simply write $E(X,Y,Z)$. Given a set $S\subseteq V(G)$, the \emph{induced subhypergraph} $G[S]$ is the $r$-graph whose vertex set is 
$S$ and whose edge set consists of all of the edges in 
$E(G)$ that have all endpoints in $S$.\\

Let $H$ be an $r$-graph and $t\in \mathbb{N}$. The $t$-\emph{blow-up} of $H$, denoted by $H(t)$, is the $r$-graph  with its vertex set partitioned in $|V(H)|$  sets $V_1,V_2,\ldots,V_{|V(H)|}$, each of size $t$ and edge set 
$\{\{a_1, \ldots, a_r\}: a_j\in V_{i_j}, j =1, \ldots, r,  \{i_1, \ldots, i_r \}\in E(H)\}$. Informally, $H(t)$ is obtained from $H$ by replacing each vertex $i$  with an independent set $V_i$ and each hyperedge $e$ of $H$ with a complete $r$-partite hypergraph with parts corresponding to the vertices of $e$. \\

We say that a $3$-graph $G$ is a \emph{weak $t$-blowup} of $H$, which we also call \emph{weak} $H(t)$, if the vertex set of $G$ can be partitioned into $|V(H)|$  sets $V_1,V_2,\ldots,V_{|V(H)|}$ each of size $t$ such that if $ijk\in E(H)$ then for every $a\in V_i, b\in V_j, c\in V_k$ we have $abc\in E(G)$, and if $ijk\not\in E(H)$ then for every $a\in V_i, b\in V_j, c\in V_k$ we have $abc\not\in E(G)$. Moreover, $V_i$ is an independent set for $i=1, \ldots, |V(H)|$.  Note that we do not impose any condition on 3-tuples of vertices with exactly two vertices in some part $V_i$.\\

Denote by $r_r(t,t)$ the \emph{Ramsey number} of $K_t^r$ versus $K_t^r$, i.e., the minimum number of vertices $m$ such that every 2-coloring of the edges of $K_m^r$ contains a monochromatic $K_t^r$. Erd\H{o}s, Hajnal and Rado~\cite{EHR} showed that there exists constants $c>0$ such that
   $ r_3(t,t)<2^{2^{ct}}.$\\

%
%
Next, we shall provide a construction of  a $(6,10)$-free graph that we shall use to provide an upper bound in Theorem \ref{610theo}. \\
\subsubsection{Construction of the $3$-graph $H_n^{\mathrm{it}}$}
Let $H$ be the $3$-graph with vertex set $[6]$ and edges  $123,$ $124,$ $345,$ $346,$ $561,$ $562,$ $135,$ $146,$ and $236$. 
Note that adding the edge $245$ to $H$ results in a $5$-regular $3$-graph on $6$ vertices, which is $K_4^{3-}$-free and the basis for the construction for the lower bound on $\pi(K_4^{3-})$ by Frankl and F\"{u}redi~\cite{FF}.\\

We define the following iterated unbalanced blow-up of this graph. Denote by $H_n$ the $3$-graph on $n$ vertices where the vertex set is partitioned into six sets $A_1, A_2, A_3, A_4, A_5, A_6$, where 
$$
|A_2|=|A_4|=|A_5|=\left\lceil \frac{n}{3\sqrt{3}} \right\rceil, \quad |A_1|=|A_3|=\left\lceil n\left(\frac{1}{3}-\frac{1}{3\sqrt{3}}\right) \right\rceil   \text{ and }  |A_6|=n\left(\frac{1}{3}-\frac{1}{3\sqrt{3}}\right)+O(1).$$
The $3$-graph $H_n$ consists of all triples $xyz$, where $x\in A_i, y\in A_j$ and $z\in A_k$ and $ijk\in E(H)$. Now, let $H_n^{\mathrm{it}}$ be the $3$-graph constructed from $H_n$ by iteratively adding a copy of $H_{|A_i|}$ with vertex set $A_i$ for all $i\in[6]$ if $|A_i|$ is sufficiently large.

\begin{lemma}\label{construction}
\label{cons610} The graph $H_n^{\mathrm{it}}$ is an $n$-vertex $3$-graph  with $\frac{4}{3+7\sqrt{3}}\binom{n}{3}+o(n^3)$ edges such that  every $6$ vertices in $H_n^{\mathrm{it}}$ induce at most $9$ edges. In particular, $H_n^{\mathrm{it}}$ is $(6,10)$-free.
\end{lemma}
We present the proof of this lemma in the appendix.

\subsection{Lemmas}

The following lemma shows that every sufficiently large $3$-graph can be partitioned into many large homogeneous sets. 
\begin{lemma}
\label{cliqueorind}
    Let $t> 0$. Then there exists $n_0=n_0(t)$ such that for every $n\geq n_0$, if $G$ is an $n$-vertex $3$-graph, then $G$ or $G^c$ contains at least $n/t-\sqrt{n}$ pairwise disjoint homogeneous sets of size $t$.
\end{lemma}
\begin{proof}
Let $t>0$ be fixed. 
Set $n_0=(\lceil 2^{2^{ct}} \rceil)^2$ and let $n\geq n_0$. Let $G=G_0$ be an $n$-vertex $3$-graph. Since $n\geq r_3(t,t)$,  there exists a homogeneous set of size $t$ in $G$. Call it $D_0$ and define $G_1=G_0\setminus D_0$. We iteratively repeat this process.  Define $G_{i+1}:=G_i \setminus D_i$, where $D_i$ is a homogeneous set of size $t$ in $G_i$.  We can proceed as long as $ |V(G_i)|> r_3(t,t).$
 Since  $ r_3(t,t)\leq \left\lceil2^{2^{ct}}\right\rceil \leq \sqrt{n_0} \leq \sqrt{n},$ we have found at least $(n-\sqrt{n})/t\geq n/t-\sqrt{n}$ pairwise disjoint homogeneous sets of size $t$ each. 
\end{proof}

The following Lemma analyses the structure \textquotedblleft between" two large vertex sets. This is partly motivated by a result by Fox and Sudakov~\cite{FS08} for $2$-graphs.

\begin{lemma}\label{ramseystuff}
Let $t\geq 0$. Then there exists $n_0$ such that for all $n\geq n_0$ the following holds. Let $G$ be a $3$-graph with vertex set $V(G)=A\cup B$ with $A \cap B = \emptyset$, $|A|=|B|=n$.
Then there exist sets $A' \subseteq A$, $B' \subseteq B$ with $|A'|=|B'|=t$ such that each of the edge sets $E(A', A', B')$ and $E(A', B', B')$
is either empty or complete.
\end{lemma}
\begin{proof}
Let $m = \underbrace{4^{4 ^{\cdots ^ {4^{t}}}}}_{2t} $, let $n_0 = \underbrace{4^{4 ^{\cdots ^ {4^{2t-1}}}}}_{m}$.
Let $A$ and $B$ be sets of size $n \ge n_0$. For $a \in A, X \subseteq B$ we define an auxiliary $2$-graph $G_a^X = (X, \binom X2)$ and an edge-coloring $c_a^X : E(G_a^X) \to \{r,b\}$ with $c_a^X(\{b_1, b_2\}) =  \begin{cases} r,& \{a,b_1, b_2\} \in E(G), \\
b, & \text{else}.\end{cases}$\\ Note that by the standard bound on the diagonal Ramsey number $r_2(s, s) \le 4^{s}$, each $2$-colored $2$-clique on $k$ vertices contains a monochromatic clique of size $\log_4(k)$.

Let $A= \{a_1, \ldots, a_n\}$, let $B_1 \subseteq B$ be the vertex set of a monochromatic clique in $G_{a_1}^B$ of size $\log_4(|B|)$. Now assume $B_i$, $i \ge 1$, has been chosen. Let $B_{i+1} \subseteq B_{i}$ be a monochromatic clique in $G_{a_{i+1}}^{B_i}$ of size $\log_4(|B_i|)$. Thus, after $m$ iterations we obtain a set $B_{m}$ of size $|B_{m}| = \underbrace{\log_4 \cdots \log_4}_{m}(n) \ge 2t -1$, such that for each $a_i$, $i \in [m]$, the set $E(\{a_i\}, B_m, B_m)$  is either empty or complete. Thus, there exists a subset $A'' \subset A, |A''| = \left\lceil\frac{m}{2}\right\rceil \ge  \underbrace{4^{4 ^{\cdots ^ {4^t}}}}_{2t-1} $, such that the set $E(A'', B_m, B_m)$ is either empty or complete.

Now we repeat this process with vertices in $B'' = B_{m}$, to obtain a subset $A' \subseteq A''$, $|A'| = \underbrace{\log_4\cdots \log_4 }_{|B''|}(|A''|) \ge t$, such that for each vertex  $b \in B''$, the set $E(A', A',\{b\} )$ is either empty or complete. Thus, there exists a subset $B' \subseteq B''$, $|B'| \ge \left\lceil\frac{|B''|}2\right\rceil = t$ such that the set $E(A', A', B')$ is either empty or complete.
The sets $A', B'$ satisfy the conditions of the lemma, completing the proof. 
\end{proof}
The next lemma shows that in a $(6,10)$-free $3$-graph there cannot be a large independent set and a large clique that are disjoint.

\begin{lemma}\label{emptycomplete}
	There exists $t_0>0$ such that for all $t\geq t_0$ the following holds. Let $G$ be a $2t$-vertex $3$-graph with vertex set $V(G)=A\cup B$ where $A\cap B=\emptyset$, $|A|=|B|=t$, $G[A]$ is a clique and $G[B]$ is an independent set. Then $G$ induces $(6,10)$. 
\end{lemma}
\begin{proof}
    By Lemma~\ref{ramseystuff}, for sufficiently large $t$, we can we find subsets $A' \subseteq A$, $B' \subseteq B$ with $|A'|=|B'|=5$ such that the two sets $E(A',A',B')$ and $E(A',B',B')$ are either empty or complete. 
	
	If $E(A',B',B')$ is complete, then any vertex from $A'$ together with the 5 vertices from $B'$ induces $(6,10)$. If $E(A',A', B')$ is empty, then any vertex from $B'$ together with the five vertices from $A'$ induces $(6,10)$. Hence, we may assume that $E(A',B',B')$ is empty and $E(A',A', B')$ is complete. But then three arbitrary vertices from $A'$ together with three arbitrary vertices from $B'$ induce $(6,10)$.
\end{proof}

Lemma~\ref{cliqueorind} together with Lemma~\ref{emptycomplete} immediately implies the following lemma.

\begin{lemma}
\label{cliqueorind2}
    There exists $t_0$ such for all $t\geq t_0$ the following holds. There exists $n_0=n_0(t)$ such that for all $n\geq n_0$, if $G$ is a $(6,10)$-free $n$-vertex $3$-graph, then either $G$ or $G^c$ contains at least $n/t-\sqrt{n}$ pairwise disjoint independent sets of size $t$.
\end{lemma}

\begin{lemma}\label{veryempty}
	Let $t'>0$. Then there exists $t_0>0$ such that for all $t\geq t_0$ the following holds. Let $G$ be a $(6,10)$-free  $2t$-vertex $3$-graph with vertex set $V(G)=A\cup B$ where $|A|=|B|=t$, $A \cap B = \emptyset$, $G[A]$ and $G[B]$ are independent sets. Then there exists $A' \subseteq A$, $B' \subseteq B$ of sizes $|A'|=|B'|=t'$ such that the two sets $E(A', B', B')$ and $E(A', A', B')$ are empty.
\end{lemma}
\begin{proof}
	We apply Lemma~\ref{ramseystuff} for $t'$. Then there exists $t_0$ such that for $t \ge t_0$, we find $A' \subseteq A, B' \subseteq B$, such that the two sets $E(A',A',B')$ and $E(A',B',B')$ are either empty or complete. Assume the set $E(A',A',B')$ is complete. Then we find induced $(6,10)$ by taking any $5$ vertices from $A'$ and $1$ vertex from $B$. By symmetry the same holds for the set $E(A',B',B')$, so in particular, $G[A' \cup B']$ is the empty graph.
\end{proof}

\begin{lemma}\label{almostdone}
There exists $t_0>0$ such that for all $t\geq t_0$ a weak $K_4^3(t)$ and also a weak $K_4^{3-}(t)$ induces $(6,10)$.
\end{lemma}
\begin{proof}
Let $G$ be a weak $K_4^{3-}(t)$ with independent sets $V_1, V_2, V_3, V_4$.
	By iteratively applying Lemma~\ref{veryempty} to all of the tuples $(V_i, V_j)$, $1 \le i < j \le 4$, we obtain an induced copy $H \subseteq G$ of $ K_4^{3-}(2)$ with sets $X_1, X_2, X_3, X_4$, $X_i \subset V_i$, $i \in [4]$, i.e., $H[X_i \cup X_j]$ is empty for all $i \neq j$, the sets $E(X_i, X_j, X_k)$ are complete for $\{i,j,k\} \in \binom{[4]}3$ except for $E(X_2, X_3, X_4)$, which is empty.
	Let $x_1, x_1' \in X_1$, $x_2, x_2' \in X_2$, $x_3 \in X_3$ and $x_4 \in X_4$. Then $\{x_1, x_1', x_2, x_2', x_3, x_4\}$ induces $(6, 10)$.
	
Now assume there is a weak $K_4^3(t)$ called $G$ with independent sets $V_1, V_2, V_3, V_4$. By iteratively applying Lemma~\ref{veryempty} to all of the tuples $(V_i, V_j)$, $1 \le i < j \le 4$, we obtain an induced copy $H \subseteq G$ of $ K_4^{3}(3)$ with sets $X_1, X_2, X_3, X_4$, $X_i \subset V_i$, $i \in [4]$, i.e., $H[X_i \cup X_j]$ is empty for all $i \neq j$ and the sets $E(X_i, X_j, X_k)$ are complete for all $\{i,j,k\} \in \binom{[4]}3$. Let $x_2 \in X_2, x_3 \in X_3, x_4 \in X_4$.
Then $H[X_1 \cup \{x_2, x_3, x_4\}]$ is a $6$-vertex $3$-graph spanning exactly $10$ edges. 
\end{proof}

\begin{lemma}
\label{findingparindK43}
Let $t>0$ be an integer and $\delta>0$. Then there exists $m_0=m_0(t,\delta)$ such that for all $m\geq m_0$ the following holds. Let $G$ be a $3$-graph on $4m$ vertices such that the vertex set of $G$ can be partitioned into four independent sets $V_1,V_2,V_3,V_4$ of size $m$ each and the number of copies of $K_4^{3-}$ with one endpoint from each of the $V_i's$ is at least $\delta m^4$. Then $G$ contains an induced copy of a weak $K_4^3(t)$ or a weak $K_4^{3-}(t)$.  
\end{lemma}
\begin{proof}
Define the auxiliary $4$-graph $H$ on $4m$ vertices where a 4-set spans an edge iff the corresponding four vertices in $G$ form a copy of $K_{4}^{3-}$. We $5$-color the edges of $H$ in the following way: An edge $\{v_1,v_2,v_3,v_4\}$ of $H$ with $v_i\in V_i$ for $i\in [4]$ is colored with $j\in[4]$ if $\{v_1,v_2,v_3,v_4\}\setminus \{v_j\}$ is not an edge in $G$, and it is colored with color $5$ if $\{v_1,v_2,v_3,v_4\}$ induces a $K_{4}^3$ in $G$.

By pigeonhole principle, there exists $(\delta/5) m^4$ edges of the same color. Erd\H{o}s~\cite{MR183654} proved that $\pi(K_4^4(t))=0$ and thus, there exists a monochromatic copy of $K_4^4(t)$ in $H$. Denote by $T$ the vertex set of this monochromatic copy. The $3$-graph $G[T]$ is a weak $K_4^3(t)$ or weak $K_4^{3-}(t)$.
\end{proof}

We will use a supersaturation result discovered by Erd\H{o}s and Simonovits~\cite{MR726456}. The proof presented below follows a proof given by Keevash (Lemma 2.1. in \cite{keevash2011hypergraph}).

\begin{lemma}
\label{supersat}
For $\varepsilon> 0$ and families $\mathcal{F},\mathcal{G}$ of $r$-graphs,  there exists constants $\delta>0$ and $n_0 > 0$ so that if $G$ is an $r$-graph on $n > n_0$ vertices with $e(G) > (\pi({}_{\mathrm{ind}}\mathcal{F},\mathcal{G}) + \varepsilon)\binom{n}{r}$, then $G$ contains at least $\delta \binom{n}{|V(H)|}$ copies of $H$ for some $H\in \mathcal{G}$, or at least $\delta \binom{n}{|V(H)|}$ induced copies of $H$ for some $H\in \mathcal{F}$.
\end{lemma}
\begin{proof}
Let $G$ be an $r$-graph on sufficiently many vertices $n$ with $e(G) > (\pi({}_{\mathrm{ind}}\mathcal{F},\mathcal{G}) + \varepsilon)\binom{n}{r}$.
Fix an integer $k\geq r$, $k\geq |V(H)|$ for all $H\in \mathcal{F} \cup \mathcal{G}$ so that $\textup{ex}(k, {}_{\mathrm{ind}}\mathcal{F},\mathcal{G})\leq \left(\pi({}_{\mathrm{ind}}\mathcal{F},\mathcal{G}) + \frac{\varepsilon}{2}\right)\binom{k}{r}$. There are at least $\frac{\varepsilon}{2}\binom{n}{k}$ $k$-sets $K\subseteq V(G)$
with $e(G[K])>(\pi({}_{\mathrm{ind}}\mathcal{F},\mathcal{G}) + \frac{\varepsilon}{2})\binom{k}{r}$. Otherwise, we would have
\begin{align*}
    \sum_{\substack{K \subseteq V(G)\\ |K|=k}} e(G[K])\leq \binom{n}{k} \left(\pi({}_{\mathrm{ind}}\mathcal{F},\mathcal{G}) + \frac{\varepsilon}{2}\right)\binom{k}{r}+\frac{\varepsilon}{2}\binom{n}{k}\binom{k}{r}=\left(\pi({}_{\mathrm{ind}}\mathcal{F},\mathcal{G}) + \varepsilon\right)\binom{n}{k}\binom{k}{r},
\end{align*}
but we also have
\begin{align*}
     \sum_{\substack{K \subseteq V(G)\\ |K|=k}} e(G[K])=\binom{n-r}{k-r}e(G)>\binom{n-r}{k-r}\left(\pi({}_{\mathrm{ind}}\mathcal{F},\mathcal{G}) + \varepsilon \right)\binom{n}{r}= \left(\pi({}_{\mathrm{ind}}\mathcal{F},\mathcal{G}\right) + \varepsilon)\binom{n}{k}\binom{k}{r},
\end{align*}
a contradiction. By the choice of $k$, each of these $k$-sets $K$ contains an induced copy of some $H\in \mathcal{F}$ or a copy of some $H\in \mathcal{G}$. By the pigeonhole principle, there exists $H_1\in \mathcal{F}$ such that at least $\frac{\varepsilon}{2(|\mathcal{F}|+|\mathcal{G}|)}\binom{n}{k}$ of these $k$-sets $K$ contain an induced copy of $H_1$, or there exists $H_2\in \mathcal{G}$ such that at least $\frac{\varepsilon}{2(|\mathcal{F}|+|\mathcal{G}|)}\binom{n}{k}$ of these $k$-sets $K$ contain a copy of $H_2$. Thus, in the first case, the number of induced copies of $H_1$ is at least 
\begin{align*}
   \frac{ \frac{\varepsilon}{2(|\mathcal{F}|+|\mathcal{G}|)}\binom{n}{k}}{\binom{n-|V(H_1)|}{k-|V(H_1)|}}=\delta \binom{n}{|V(H_1)|}, \quad \text{for} \quad  \delta=\frac{\varepsilon}{2(|\mathcal{F}|+|\mathcal{G}|)\binom{k}{|V(H_1)|}}.
\end{align*}
Similarly, in the second case, the number of copies of $H_2$ is at least 
\begin{equation*}
\delta \binom{n}{|V(H_2)|} \quad \text{for} \quad  \delta=\frac{\varepsilon}{2(|\mathcal{F}|+|\mathcal{G}|)\binom{k}{|V(H_2)|}}. \qedhere
\end{equation*}
\end{proof}


\subsection{Proof of Theorem~\ref{610theo}.}
\begin{proof}[Proof of Theorem~\ref{610theo}]
Let $\varepsilon>0$. Fix an integer $t$ whose existence is guaranteed by Lemma~\ref{almostdone}, such that every weak $K_4^3(t)$ and also every weak $K_4^{3-}(t)$ induces $(6,10)$, see the paragraph before Lemma~\ref{almostdone} for the definition of a weak blow-up. Fix $\delta>0$ and $n_1\in \mathbb{N}$, given by Lemma~\ref{findingparindK43}, such that every $(6,10)$-free $3$-graph $G$ on $n\geq n_1$ vertices satisfying $e(G)\geq (\pi({}_{\mathrm{ind}}\mathcal{F}_6^{10},\{K_{4}^{3-}\})+\varepsilon)\binom{n}{3}$ contains at least $2\delta \binom{n}{4}$ copies of $K_4^{3-}$. Let $m_0=m_0(t,\delta)$ be given by Lemma~\ref{findingparindK43}. Fix integers $m_1$ and $n_2$ whose existence is guaranteed by Lemma~\ref{cliqueorind2}, such that $m_1\geq m_0$ and for all $n\geq n_2$, if $G$ is $(6,10)$-free $n$-vertex $3$-graph, then either $G$ or $G^c$ contains at least $n/m_1-\sqrt{n}$ pairwise disjoint independent sets of size $m_1$. Choose $n_0:=\max\{n_1,n_2,m_1^2,\lceil40000\delta^{-2}\rceil\}$ and let $n\geq n_0$. 

Let $G$ be a $(6,10)$-free $n$-vertex $3$-graph satisfying the density assumption \eqref{edge density}:
\begin{align}\nonumber
    \frac{e(G)}{\binom{n}{3}}\in \left[\pi({}_{\mathrm{ind}}\mathcal{F}_6^{10},\{K_4^{3-}\})+\varepsilon , 1-\pi({}_{\mathrm{ind}}\mathcal{F}_6^{10},\{K_4^{3-}\})-\varepsilon  \right].
\end{align}
By Lemma~\ref{cliqueorind2} either $G$ or $G^c$ contains at least $n':=n/m_1-\sqrt{n}$ pairwise disjoint independent sets, each of size $m_1$. Since the density assumption is symmetric, and since $G$ induces $(6,10)$ if and only if $G^c$ induces $(6,10)$, we can assume, without loss of generality, that $G$ contains at least $n'$ pairwise disjoint independent sets $V_1,V_2,\ldots, V_{n'}$ of size $m_1$ each.

By Lemma~\ref{supersat}, $G$ contains at least $2\delta \binom{n}{4}$ (not necessarily induced) copies of $K_4^{3-}$. 
We call a $4$-set \emph{transversal} in $G$ if each of the four vertices is in a different $V_i$. A copy of $K_4^{3-}$ in $G$ is called \emph{transversal} if the vertex set of the copy is transversal in $G$. The number of $4$-sets which are not transversal in $G$ is at most
\begin{equation}\nonumber
    \sqrt{n}n^3+n'\binom{m_1}{2}n^2\leq n^{\frac{7}{2}}+m_1n^3 \leq 2n^{\frac{7}{2}}, 
\end{equation}
for $n\geq m_1^2$. The number of transversal copies of $K_4^{3-}$ in $G$ is at least $\frac{3}{2}\delta \binom{n}{4}$, since
\begin{align*}
    2\delta \binom{n}{4}-\frac{3}{2}\delta \binom{n}{4}=\frac{\delta}{2} \binom{n}{4}\geq \frac{\delta}{2} \frac{n^4}{2\cdot 4!}=\frac{\delta}{96}n^4> 2n^{7/2}, 
\end{align*}
where the last inequality holds for $n\geq 40000\delta^{-2}$.
By pigeonhole principle there exist $1\leq i_1<i_2<i_3<i_4\leq n'$, such that the number of copies of $K_4^{3-}$ with one endpoint in each of $V_{i_1},V_{i_2},V_{i_3},V_{i_4}$ is at least 
\begin{align*}
\frac{\frac{3}{2}\delta \binom{n}{4}}{\binom{n'}{4}}\geq \frac{\delta \frac{n^4}{4!}}{\frac{\left(\frac{n}{m_1}\right)^4}{4!}} =\delta m_1^4.
\end{align*}
By Lemma~\ref{findingparindK43}, the $3$-graph $G[V_{i_1}\cup V_{i_2}\cup V_{i_3} \cup V_{i_4}]$ contains a weak $K_4^{3-}(t)$ or a weak $K_4^3(t)$ as an induced subhypergraph. This contradicts Lemma~\ref{almostdone}.

We conclude $\sigma_3(6,10)\geq 1-2\pi({}_{\mathrm{ind}}\mathcal{F}_6^{10},\{K_4^{3-}\})$. In fact, $\sigma_3(6,10)= 1-2\pi({}_{\mathrm{ind}}\mathcal{F}_6^{10},\{K_4^{3-}\})$ holds by the following argument: Let $G$ be an $n$-vertex $K_4^{3-}$-free and $(6,10)$-free $3$-graph with exactly $\textup{ex}(n,{}_{\mathrm{ind}}\mathcal{F}_6^{10},\{K_4^{3-}\})$ many edges. Since $G$ is $K_4^{3-}$-free, every four vertices span at most $2$ edges, so using double counting, we see that every  $6$ vertices span at most $\binom{6}{4}\cdot 2/3 =10$ edges. Since $G$ is also $(6,10)$-free, every $6$ vertices span only at most $9$ edges. We conclude that every subgraph $G' \subseteq G$ is $(6,10)$-free. Further, by symmetry, also the complement $3$-graph of any $G' \subseteq G$ is $(6,10)$-free. This proves the first part of the theorem. \\

To get specific numerical  bounds on the forcing density, recall again that  if 
$$ \frac{e(G)}{\binom{n}{3}}\in \left[\pi({}_{\mathrm{ind}}\mathcal{F}_6^{10},\{K_4^{3-}\})+\varepsilon , 1-\pi({}_{\mathrm{ind}}\mathcal{F}_6^{10},\{K_4^{3-}\})-\varepsilon  \right],$$
then $G$ induces $(6,10)$. In particular, if $ \frac{e(G)}{\binom{n}{3}}\in  \left[\pi(K_4^{3-})+\varepsilon , 1-\pi(K_4^{3-})-\varepsilon  \right]$, 
then $G$ induces $(6,10)$. 
The Tur\'an density of $K_4^{3-}$ is not known precisely. The best currently known bounds on the Tur\'an density of $K_4^{3-}$ are $0.28571 \approx \frac{2}{7}\leq \pi(K_4^{3-}) \leq 0.28689$, where the lower bound construction was given by Frankl and F\"{u}redi~\cite{FF}. The upper bound was proved by Vaughan~\cite{flagmatic} who applied the flag algebra method, see also the webpage of Lidick\'{y}~\cite{flagmatic2}. 
Thus $\sigma_3(6,10)\geq 1 - 2\cdot  0.28689 = 0.42622.$
However, from Lemma \ref{construction}, we have that  there is a $3$-graph on $n$ vertices and $\frac{4}{3+7\sqrt{3}}\binom{n}{3}(1+o(1))$ hyperedges, such that each of its subgraphs is $(6,10)$-free. Moreover, the complement of this $3$-graph has 
$\left(1-\frac{4}{3+7\sqrt{3}}\binom{n}{3}\right)(1+o(1))$ hyperedges and each of its supergraphs is $(6,10)$-free. Thus $\sigma_3(6,10)\leq 1 - 2\frac{4}{3+7\sqrt{3}} = 0.47105$.\end{proof}

%
                                                \section{Proof of Theorem \ref{Diophantine}}
\label{sec:Thm2}
\subsection{Constructions and notations}

We shall first construct a special class of $3$-graphs.\\

Let $n, k \in \mathbb N$, $k \le n$ and  $S \subseteq [2]$. Let $G(S,n,k)$ be the $3$-graph with vertex set $A\cup B$, $|A| = k$, $|B| = n-k$, where $A$ and $B$ are disjoint such that $A$ induces a clique,  $B$ induces an independent set, called {\it base set},  and we have the additional edges  $\bigcup_{i \in S} E_i$, where $E_i = \{A' \cup B' : A' \in \binom Ai, B' \in \binom B{3-i}  \}$.
Thus, $G_\emptyset(n,k)$ is just  a clique on $k$ vertices and $n-k$ isolated vertices, and $G_{[2]}(n,k)$ is the complete graph on $n$ vertices with a clique of size $n-k$ removed. For an illustration of $G(\{2\},n,k)$ see Figure~\ref{fig:G2nk}.

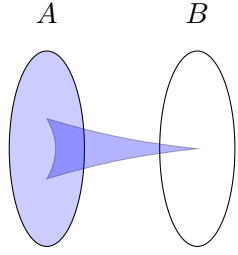
\begin{figure}[h!]
\begin{center}
\begin{tikzpicture}[scale=2.0]
\draw[fill=blue!20] (0,0) coordinate(x1) ellipse (0.25cm and 0.65cm);
\draw (1,0) coordinate(x1) ellipse (0.25cm and 0.65cm);
\draw
(0,0.2) coordinate(1) 
(0,-0.2) coordinate(2) 
(0,0.3) coordinate(3) 
(1,0.2) coordinate(4) 
(1,0) coordinate(5) 
(1,0.4) coordinate(6) 
;
\draw[hyperedge] (1) to[bend left] (2) to[bend left=5] (5) to[bend left=5] (1);
\draw (0,0.9) node{$A$};
\draw (1,0.9) node{$B$};
\end{tikzpicture}
\caption{Illustration of $G(\{2\},n,k)$.}
\label{fig:G2nk}
\end{center}
\end{figure}

Note that the complement of $G(S, n, k)$ is $G([2]-S, n, n-k)$.  Let $f(S,n,k) = |E(G(S,n,k))|$.   
We call a $3$-graph $G$ \emph{m-sparse} if every subset of $m$ vertices in $G$ induces at most $m$ edges. 
We say that a $3$-graph $G$ is  {\it canonical plus with parameters $(S,n,k)$}, or simply {\it canonical plus} if $G$ is a $3$-graph obtained as a union of $G(S,n,k)$ and an $m$-sparse graph whose vertex set is the base independent set of $G(S,n,k)$.  A $3$-graph $G$ is  {\it canonical minus with parameters $(S,n,k)$}, or simply {\it canonical minus}, if $G$ is the  complement of a canonical plus graph with parameters $([2]-S,n, n-k)$. 
Note that a canonical minus graph with parameters $(S, n,k)$ is obtained from the graph $G(S,n, k)$ by removing edges of a copy of an $m$-sparse graph from the clique $A$.
We see that (letting $\binom{y}{x}=0$ for $y<x$), that  $$f(S,n,k) = \binom{k}{3}  + \sum\limits_{i \in S} \binom{k}{i}\binom{n-k}{3-i}.$$

Moreover,  $|f(S,n,x)-f(S,n,x-1)| \in O(n^{2})$. 
Note that any induced subgraph of a canonical plus $3$-graph with parameters $(S, n, k)$ is a canonical plus $3$-graph with parameters $(S, n', k')$, for some $n'$ and $k'$.  
A similar statement holds for canonical minus graphs. Thus, these two classes of graphs are hereditary.
%
We see that if an $m$-vertex $3$-graph is canonical  plus with parameters $(S, m, x)$, then the number  of edges in such a graph  is in  the interval  $[f(S,m,x), f(S,m,x) + m]$.
Similarly, the number  of edges in a canonical  minus graph  with parameters $(S, m, x)$  is in the interval $ [f(S,m,x)-m, f(S,m,x)] $.  Thus, if $f$ is the number of edges of a graph that could be represented as both a canonical plus and a canonical minus graph with first parameter $S$ and $m$ vertices, then $f\in F(S,m)$,  where $$F(S, m)= \bigcup\limits_{x=0}^{m-1} [f(S,m,x), f(S,m,x) + m] \cap \bigcup\limits_{x=1}^m [f(S,m,x)- m, f(S,m,x)]\subseteq \left\{0,1,\ldots, \binom{m}{3} \right\}.$$

\subsection{Proof idea}

\noindent
We are using the following general principle:

\begin{prop}
Let $\cC_1, \ldots, \cC_k$ be hereditary classes of $r$-graphs such that for any $c$,  $0<c<1/2$, any sufficiently large $n$, and any $e$ with $c \binom nr \leq e \leq (1-c)\binom nr$, there is a graph $G_i\in \cC_i$ on $n$ vertices and $e$ edges for all $i=1, \ldots, k$. If for any sufficiently large $n$ and some $i\in [k]$, each $n$-vertex graph in $\cC_i$ is $(m,f)$-free, then $\sigma_r(m,f)=0$.
\end{prop}

Here, we use two classes $\cC_1$ and $\cC_2$ of $3$-graphs that are canonical plus and canonical minus with the same first parameter $S$. 
Specifically, the main idea of the proof of Theorem \ref{Diophantine} is that for any sufficiently large $n$,  any $S\subseteq [2]$, and any $e$ in the interval $[c\binom n 3, (1-c)\binom n3]$ for $0<c<1/2$, there is a canonical plus $3$-graph $G^+_{c,S}$ and a canonical minus $3$-graph $G^-_{c,S}$ with first parameter $S$, on $n$ vertices and $e$ edges. 
If, for a pair $(m,f)$, $f \not\in F(S,m)$ for some $S \subseteq [2]$, then the pair $(m,f)$ is not representable as a canonical plus or canonical minus graph with first parameter $S$.
Then in particular,  $G^+_{c,S}$ and $G^-_{c,S}$  are $(m,f)$-free and  $(n, e) \not\to (m,f)$. Letting $c$ be arbitrarily small, we 
conclude that $\sigma_3(m,f)=0$ for such a pair $(m,f)$. Finally, we derive number theoretic conditions for a pair $(m,f)$ not being representable by a canonical plus or a canonical minus  graph.

\subsection{Lemmas}
In the following lemmas, $n,m,f,e$ are non-negative integers with $m > 3$, $0 < f < \binom m3$.  In \cite{W22} it was shown that for any $m\leq 15$ and for any $0<f<\binom m3$ such that $(m,f)\neq (6,10)$, $\sigma_3(m,f)=0$.  Thus, we can assume that $m\geq 16$. The following folklore result can be obtained by a standard probabilistic argument. 
\begin{lemma}\label{lem:probabilistic}
Let $m>0$. Then for any sufficiently large $n$ there exists an $n$-vertex $3$-graph with $\Omega(n^{2+\frac1{m+1}})$ edges which is  $m$-sparse.\end{lemma}
For a proof of Lemma~\ref{lem:probabilistic} see e.g. \cite{W22}. The next lemma is a generalization of a similar statement proven in \cite{EFRS} for graphs. 
\begin{lemma}\label{lem:adding_clique}
	Let  $S\subseteq [2]$ and $c$ be a constant, $0< c <1/2$. For $n \in \mathbb N$ sufficiently large and any  $e$ where $c<e <(1-c)\binom{n}{3}$, 
	there exist  $3$-graphs $G_1(n,e)$ and $G_2(n,e)$  on $n$ vertices and $e$  edges  that are  canonical plus  and canonical minus respectively,  with first  parameter $S$. 
		\end{lemma}

\begin{proof}
Let $n$ be a given sufficiently large integer. Let $k$  be a non-negative integer such that either  $f(S,n,k)\le e \le f(S,n,k+1)$ or $f(S,n,k)\le e \le f(S,n,k-1)$  holds. Without loss of generality assume that $f(S,n,k)\le e \le f(S,n,k+1)$. Let $c_1=1-c$.
     Note that since $e \leq c_1\binom{n}{3}$, $\binom{k}{3}  \leq  c_1\binom{n}{3}$, we have $k \leq \sqrt[3]{c_1}n +1\leq c'n$, where $c'<1$ is a constant.  \\
 
 Let $G'$ be an $m$-sparse $3$-graph on $n-k$ vertices with $|E(G')| \geq (n-k)^{2+\frac{1}{m+1}}$.
  The existence of $G'$ is guaranteed by Lemma~\ref{lem:probabilistic}.  Define $G''$ to be the $3$-graph obtained as a union of $G(S,n,k)$ and a copy of $G'$ on the vertex set that is the base independent set of $G(S,n,k)$. 	
	Then $|E(G'')| \geq f(S,n,k)+ (n-k)^{2+\frac{1}{m+1}}  \geq f(S,n,k+1) \geq e$.
	Here, the second inequality holds since $f(S,n,k+1)-f(S,n,k)= O(n^{2})$.
	Finally, let $G_1(n,e)$ be a subgraph of $G''$ with $e$ edges, obtained from $G''$ by removing some edges of $G'$. 	\\
	
	For the second part of the lemma, take $G_2(n,e)$ to be the complement of $G_1(n, \binom n3 - e)$ with first parameter $[2]-S$,  guaranteed by the first part of the lemma.
	\end{proof}

\begin{lemma}\label{F}
    Let $S \subseteq [2]$. If $f \not\in F(S, m)$, 
    then $\sigma_3(m,f) = 0$.
\end{lemma}

\begin{proof}
Assume we have integers $m, f$ as above, some $S \subseteq [2]$ and $f \not\in F(S,m)$.
Let $c$ be a constant, $0 < c < 1/10$,  $n \ge n_0$, and  $e$ be any integer satisfying  $c\binom n3 \leq  e \le (1-c) \binom n3$. 
Define graphs $G_1=G_1(n, e)$ and $G_2=G_2(n, e)$ whose existence is guaranteed by Lemma \ref{lem:adding_clique}. 
Any induced subgraph of $G_1$ on $m$ vertices is canonical plus with parameters $(S, m, x)$ for some $x$ and thus, its number of edges is in $\bigcup\limits_{x=0}^{m-1} [f(S,m,x), f(S,m,x) + m]$.
Any induced subgraph of $G_2$ on $m$ vertices is canonical minus with parameters $(S, m, x)$ for some $x$ and thus,  its number of edges is in $ \bigcup\limits_{x=1}^m [f(S,m,x)- m, f(S,m,x)]$.
Since $f\not\in F(S,m)$, we get that $G_1$ and $G_2$ are $(m,f)$-free.
Letting $c$ go to zero, we see that  $\sigma_3(m,f) = 0$.
\end{proof}
\vskip 0.5cm

In the following lemmas we shall use the set $S=\emptyset$, $S=\{1\}$, or $S=\{2\}$, to claim that for many pairs $(m,f)$, $\sigma_3(m,f)=0$.

\begin{lemma}\label{cor:small_f}
Let $m\geq 7$ and $0 < f < \binom{m-1}2$. Then $\sigma_3(m,f) = 0$. 
\end{lemma}
\begin{proof}
Let $S = \{1\}$. By Lemma \ref{F}, it is sufficient to verify that $f \not\in  F(\{1\}, m)$. For that it is sufficient to check that $F(\{1\}, m)\cap [1, \binom {m-1}{2}-1]=\emptyset$.    Recall that $$  F(\{1\},m) =   \bigcup\limits_{x=0}^{m-1} [f(\{1\},m,x), f(\{1\},m,x) + m] \cap \bigcup\limits_{x=1}^m [f(\{1\},m,x)- m, f(\{1\},m,x)].$$ 
Note that $f(\{1\},m,0) = 0$, $f(\{1\},m, 1) =  \binom{m-1}{2} $, and $f(\{1\},m, x) \geq  \binom{m-1}{2}$, for  $x >  1$. 
Thus,  we have 
\begin{align*} F(\{1\},m) \cap [1, \tbinom{m-1}{2}-1] &= \bigcup\limits_{x=0}^{m-1} [f(\{1\},m,x), f(\{1\},m,x) + m]  \cap [1, \tbinom{m-1}{2}-1] \\
	&= [f(\{1\},m,0), f(\{1\},m,0) + m] \cap [1, \tbinom{m-1}{2}-1]=[1,m], \end{align*} 
and 
\begin{align*}\bigcup\limits_{x=1}^m [f(\{1\},m,x)- m, f(\{1\},m,x)] \cap [1, \tbinom{m-1}{2} -1] &=  [f(\{1\},m,1)- m, f(\{1\},m,1) - 1]\\
&=[\tbinom{m-1}{2}-m, \tbinom{m-1}{2}-1].
\end{align*}
In particular, we have
\begin{align*}  F(\{1\},m) \cap [1, \tbinom{m-1}{2} -1]
 = [0, m] \cap [\tbinom{m-1}{2}-m, \tbinom{m-1}{2}-1]  = \emptyset,
\end{align*}
where in the last step we used that $\binom{m-1}{2} > 2m$. Thus,  $\sigma_3(m,f) = 0$. 
\end{proof}

\begin{lemma}\label{cor:s_empty}
Let $f$ be an integer such that $ \binom{m-1}2 \leq f<\binom m3$ and  for any $x \in [m]$,  $f \neq  \binom x3$. Then  $\sigma_3(m,f) = 0$.
\end{lemma}
\begin{proof}
Define $f$ as given in the statement of the lemma and $S=\emptyset$. By Lemma \ref{F},  it is sufficient to prove that $ f\not\in F(\emptyset, m)$ and in particular it is sufficient to show that  $F(\emptyset,m)\cap [\binom{m-1}2, \binom m3-1]\subseteq \{\binom x3: ~x\in [m]\}$.
Since $f(\emptyset,n,x) = \binom x3$, we have    
$$ F(\emptyset,m) =  \bigcup\limits_{x=0}^{m-1} [\tbinom x3, \tbinom x3 + m] \cap \bigcup\limits_{x=1}^m [\tbinom {x}3 -m, \tbinom{x}{3}],$$
see Figure \ref{figureintervals} for an illustration of the set $F(\emptyset, m)$.
Note that $\binom x3 \ge \binom{m-1}{2}$ implies $\binom{x}{2} > 2m$, which is equivalent to 
$\binom{x+1}{3} - m > \binom{x}{3} + m$. In particular, in this case the interval $[\binom x3, \binom{x+1}3]$ is long enough that we have $[\binom x3, \binom x3 + m] \cap [\binom {x'}3 -m, \binom{x'}3] = \emptyset$ for $x \neq x'$ and $\binom x3, \binom{x'}3 \ge \binom{m-1}2$.
Thus,  $$F(\emptyset,m) \cap [\tbinom{m-1}{2}, \tbinom m3] 
\subseteq \{ \tbinom x3 : x \in [m]\}. \qedhere$$ 
\end{proof}
	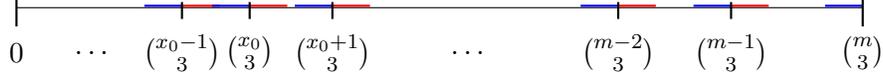
\begin{figure}[H]
	\centering
	\begin{tikzpicture}[scale=.5]
		\draw[] (0,0) -- (22.5,0) ; 
		
		\draw[red, thick] (4.4, 0.05) -- (5.4,0.05);
		\draw[red, thick] (6.2, 0.05) -- (7.2,0.05);
		\draw[red, thick] (8.4, 0.05) -- (9.4,0.05);
		\draw[red, thick] (16, 0.05) -- (17,0.05);
		\draw[red, thick] (19, 0.05) -- (20,0.05);
		
		\draw[blue, thick] (4.4, 0.05) -- (3.4,0.05);
		\draw[blue, thick] (6.2, 0.05) -- (5.2,0.05);
		\draw[blue, thick] (8.4, 0.05) -- (7.4,0.05);
		\draw[blue, thick] (16, 0.05) -- (15,0.05);
		\draw[blue, thick] (19, 0.05) -- (18,0.05);
		\draw[blue, thick] (22.5, 0.05) -- (21.5,0.05);
		
		\node[] at (0,-1.2) {$0$};	
		\node[] at (2,-1.2) {$\ldots$};
		\node[] at (12,-1.2) {$\ldots$};
		\node[] at (8.4,-1.2) {$\binom{x_0+1}{3}$};
		\node[] at (6.2,-1.2) {$\binom{x_0}{3}$};
		\node[] at (4.4,-1.2) {$\binom{x_0-1}{3}$};
		\node[] at (16,-1.2) {$\binom{m-2}{3}$};
		\node[] at (19,-1.2) {$\binom{m-1}{3}$};
		\node[] at (22.5,-1.2) {$\binom{m}{3}$};

		\draw[thick, black] (0,0.3) -- (0,-0.5);
		\draw[thick, black] (4.4,0.15) -- (4.4,-0.3);
		\draw[thick, black] (6.2,0.15) -- (6.2,-0.3);
		\draw[thick, black] (8.4,0.15) -- (8.4,-0.3);
		\draw[thick, black] (16,0.15) -- (16,-0.3);
		\draw[thick, black] (19,0.15) -- (19,-0.3);
		\draw[thick, black] (22.5,0.3) -- (22.5,-0.5);
	
	\end{tikzpicture}
	\caption{This figure displays the set $\bigcup\limits_{x=0}^{m-1} [\tbinom x3, \tbinom x3 + m]$ in red and the set $\bigcup\limits_{x=1}^m [\tbinom {x}3 -m, \tbinom{x}{3}]$ in blue on the number line. Here, $x_0$ is the smallest integer $x$ such that $\binom{x+1}{3} - m > \binom{x}{3} + m$.}
		\label{figureintervals}
\end{figure}

\begin{lemma}\label{cor:s_2}
Let $ m \ge 13$ and $f$ be an integer, such that  $\binom{m-1}2 \le f\le \binom m3 - \binom{m-1}{2}$ and for any   $x \in [m]$,   $ f \neq \binom x3 + \binom x2(m-x)$. Then  $\sigma_3(m,f) = 0$. 
\end{lemma}

\begin{proof}
Consider $m$ and $f$ as given in the statement of the lemma and let $S = \{2\}$.  
By Lemma \ref{F},  it is sufficient to prove that $ f\not\in F(S, m)$ and in particular, it is sufficient to show that  $$F(\{2\},m)\cap [\binom{m-1}2, \binom m3 - \binom{m-1}{2}] \subseteq \{ \binom x3 + \binom x2(m-x): ~x\in [m]\}.$$ Recall that 
$$F(\{2\},m) = \bigcup\limits_{x=0}^{m-1} [f(\{2\},m,x), f(\{2\},m,x) + m] \cap \bigcup\limits_{x=1}^m [f(\{2\},m,x)- m, f(\{2\},m,x)].  $$
From the definition of $f$, we have that $f(\{2\},m,x)= \binom x3 + \binom x2(m-x)$.
Note that for $x<4$ we have $f(\{2\}, m, x) +m < \binom{m-1}{2}$ and 
for $x>m-4$, $f(\{2\}, m, x) -m > \binom m3 - \binom{m-1}2$.
Therefore it is sufficient to consider only 
$$\bigcup\limits_{x=4}^{m-4} [f(\{2\},m,x), f(\{2\},m,x) + m] \cap \bigcup\limits_{x=4}^{m-4} [f(\{2\},m,x)- m, f(\{2\},m,x)].$$  
 One can verify, that for $m \ge 13$  and  $4 \le x \le m-4$, 
$ f(\{2\},m,x) - f(\{2\},m, x-1)>2m$.
Thus,
\begin{eqnarray*}
 \bigcup\limits_{x=4}^{m-4} [f(\{2\},m,x), f(\{2\},m,x) + m] \cap \bigcup\limits_{x=4}^{m-4} [f(\{2\},m,x)- m, f(\{2\},m,x)] \\
 = \{ f(\{2\}, m, x) : 4 \le x \le m-4\}. 
\end{eqnarray*}

 In particular, we have 
$$F(\{2\},m) \cap [\tbinom{m-1}2, \tbinom{m}{3} - \tbinom{m-1}2] \subseteq \{\tbinom x3 + \tbinom x2(m-x) : 4 \le x \le m-4\}.$$
\end{proof}

\subsection{Proof of Theorem \ref{Diophantine}}
\begin{proof}
For $m\le 15$ it was already shown in \cite{W22}, that the only possible pair $(m,f)$ with $0<f<\binom{m}{3}$ and  $\sigma_3(m,f)> 0$ is $(6,10)$, where $10 = \binom{5}{3} = \binom{6}{3} - \binom 53 = \binom 33 + \binom32(6-3)$.
Now let $m > 15$, and assume that for some $f$ we have $\sigma_3(m,f) > 0$.
Then applying Lemma \ref{cor:small_f} to $(m,f)$ and $(m, \tbinom{m}3 - f)$, we obtain that $\binom {m-1}{2} \le f \le \binom m3 - \binom {m-1}2$. 
Applying Lemma \ref{cor:s_empty} to $(m,f)$ gives us  that $f=\binom{x_1}{3}$, for some $x_1$;  applying it again to $(m, \tbinom{m}3 - f)$ gives us that  $f = \binom m3 - \binom{x_2}{3}$, for some $x_2$. Lemma~\ref{cor:s_2} shows the existence of some $x_3$, for which we have $f=\binom{x_3}3 + \binom{x_3}2(m-x_3).$ This completes the proof.
\end{proof}

\section{Concluding Remarks}
\label{sec:conclu}
In this paper we investigate $3$-uniform hypergraphs and forcing densities $\sigma_3(m,f)$. We show that $\sigma_3(6,10)>0$ and provided more specific bounds.  Apart from the pairs $(m,0)$, $(m, \binom{m}{3})$, the pair $(6,10)$  is the only known non-trivial pair for which the forcing density is positive. 
We conjecture that $(6,10)$ is the unique pair $(m, f)$ with $0<f<\binom{m}{3}$ for which $\sigma_3(m,f)>0$.\\

Theorem~\ref{Diophantine} implies that if there is no $m\neq 6$ for which there is a solution $(x_1, x_2, x_3)$, $x_i\in [m-1]$, of the system of Diophantine equations
\begin{equation}\label{dioph}
\binom{x_1}{3} = \binom{m}{3}- \binom{x_2}{3} = \binom{x_3}{3} + \binom{x_3}{2}(m-x_3),
\end{equation}  then Conjecture~\ref{conj610} is true. However, we do not know much about  solutions $(x_1, x_2, x_3)$ to the above system of equations. A computer search for suitable solutions of  (\ref{dioph}) for any given  $m\leq 10^6$ did not give a result. Considering only the equation  $\binom{x_1}{3} = \binom{m}{3}- \binom{x_2}{3} $, Sierpi\'nski \cite{S62} found an infinite class of solutions. \\

It might be possible to find stronger necessary conditions for a pair to have positive forcing density using different constructions than the ones used in the proof of Theorem~\ref{Diophantine}. In particular, the reader might wonder why Lemma~\ref{F} and the corresponding constructions in Lemma~\ref{lem:adding_clique} were not used when $S=\{1\}$. The reason for this is that the respective  function $f(\{1\},m,x)=\binom{x}{3}+x\binom{m-x}{2}$ is not monotone, making it difficult to capture the structure of the set $F(\{1\},m)$. However, this construction could very well be used to conclude that certain pairs $(m,f)$ have forcing density zero. \\

Determining the exact value of $\sigma_3(6,10)$ remains open. We believe that the upper bound from Theorem~\ref{610theo}, coming from the iterated construction $H_n^{\mathrm{it}}$ in Lemma~\ref{construction},  is tight.
\begin{conj}
\label{conj610exact}
We have $\sigma_3(6,10)=1-2\frac{12}{9+21\sqrt{3}}\approx 0.47105$.
\end{conj}
We remark that a standard flag algebra calculation yields that $\pi({}_{\mathrm{ind}}\mathcal{F}_{6}^{10},\{K_{4}^{3-}\})\leq 0.275<2/7$. Using the first part of  Theorem~\ref{610theo}, 
this gives $\sigma_3(6,10)\geq 0.45$ which improves the lower bound on  $\sigma_3(6,10)$ given in the second part of Theorem~\ref{610theo}.

\section*{Acknowledgements} We thank Christian Winter and Kevin Ford for discussions on parts of the project. Further, we thank Bernard Lidick\'y for computational assistance.

\section*{Appendix}
In this appendix, we prove Lemma~\ref{construction}.
\begin{proof}[Proof of Lemma~\ref{construction}]
We have 
\begin{align*}
    |E(H_n)|= 3 \left(\frac{n}{3\sqrt{3}}\right)^2\left(\frac{1}{3}-\frac{1}{3\sqrt{3}}\right)n+6 \left(\frac{n}{3\sqrt{3}}\right)\left(\frac{1}{3}-\frac{1}{3\sqrt{3}}\right)^2n^2+o(n^3) = \frac{2\sqrt{3}}{81}n^3+o(n^3).
\end{align*}
Since $H_n^{\mathrm{it}}$ is an $n$-vertex $3$-graph, it has at most $\binom{n}{3} \leq n^3/6$ edges. Let  $|E(H_n^{\mathrm{it}})|=dn^3+o(n^3)$ for some $d\in [0,\frac{1}{6}]$. We have 
\begin{align*}
    |E(H^{\mathrm{it}}_n)|&= \frac{2\sqrt{3}}{81} n^3 + 3d \left(\frac{n}{3\sqrt{3}}\right)^3+ 3d\left(\frac{1}{3}-\frac{1}{3\sqrt{3}}\right)^3n^3 +o(n^3)\\&=\left(\frac{2\sqrt{3}}{81} +\frac{d}{9}(2-\sqrt{3}) \right) n^3+o(n^3).
\end{align*}
Comparing the two expressions for $|E(H^{\mathrm{it}}_n)|$, we get  $d=2/(9+21\sqrt{3})$. In particular, 
\begin{align*}
    \frac{|E(H_n^{\mathrm{it}})|}{\binom{n}{3}}=\frac{4}{3+7\sqrt{3}}+o(1)\approx 0.26447 +o(1).
\end{align*}

Next we show that every  set of six vertices in $H_n^{\mathrm{it}}$ spans at most $9$ edges. Recall that $H_n^{\mathrm{it}}$ is obtained as an iterated blow-up construction with a  \textquotedblleft seed" graph $H$, where  $H$ is the $3$-graph with vertex set $[6]$ and edges  $123,$ $124,$ $345,$ $346,$ $561,$ $562,$ $135,$ $146,$ and $236$. At the first iteration, the vertices $1, \ldots, 6$ or $H$ correspond to parts $A_1, \ldots, A_6$.   We have that $H$ has three vertices of degree $4$ and three vertices of degree $5$, and $H$ is $K_4^{3-}$-free, so every subset of four vertices spans at most two edges. Moreover, the link graph of any vertex of $H$ is a subgraph of a $5$-cycle. Here, the link graph of a vertex $x$ is a 2-graph which has contains an edge $yz$ if and only if $xyz$ is an edge of $H$.\\

Let $X$ be an arbitrary set of six vertices of $H_n^{\mathrm{it}}$.\\

\textbf{Case 1:} $X$ contains vertices from six distinct parts $A_{1}, \ldots A_6$.\\ Then $|X|$ induces a copy of $H$, i.e., exactly $9$ edges. \\

\textbf{Case 2:} $X$ contains vertices from five distinct parts, say $A_{i_1}, A_{i_2}, A_{i_3}, A_{i_4}, $ and $ A_{i_5}$.\\
Assume we have two vertices in $A_{i_1}$, and one vertex in each of $A_{i_2}, A_{i_3}, A_{i_4}, $ and $A_{i_5}$.  Note that  $A_{i_1}, A_{i_2}, A_{i_3}, A_{i_4}, $ and $ A_{i_5}$   correspond to the vertices $i_1, \ldots, i_5 \in V(H)$.   Let $H' = H[\{i_1, i_2, i_3, i_4, i_5\}]$. Since the link graph of any vertex in $H$ is a subgraph of $C_5$, the link graph of any vertex in  $H'$ has at most three edges, so the maximum degree of $H’$ is at most three. This implies that the total number of edges in $H’$ is at most $3\cdot 5/3= 5$.
Since the subgraph of $H_n^{\mathrm{it}}$ induced by $X$ corresponds to $H'$ with an added copy of $i_1$ which contributes at most three edges, $X$ induces at most $5+3=8$   edges. \\


\textbf{Case 3:} $X$ contains vertices from four distinct parts:  $A_{i_1}, A_{i_2}, A_{i_3}, $ and $A_{i_4}$.\\
\textbf{Case 3.1:} $X$ contains $3$ vertices from $A_{i_1}$ and one vertex from each of $A_{i_2}, A_{i_3}, A_{i_4}$. Then $H[\{i_1, i_2, i_3, i_4\}]$ contains at most two edges, so $X$  induces at most $2\cdot 3$ edges between the parts and at most one additional edge inside $A_{i_1}$, so in total at most $7$ edges. \\
\textbf{Case 3.2:} $X$ contains two vertices from each of $A_{i_1}, A_{i_2}$ and one vertex in each of $A_{i_3}, A_{i_4}$. Again, since $H[\{i_1, i_2, i_3, i_4\}]$ contains at most two edges, $X$ induces at most $2\cdot 4 = 8$ edges.\\

\textbf{Case 4:} $X$ contains vertices from three distinct parts: $A_{i_1}, A_{i_2}$, and $ A_{i_3}$.\\ If we have two vertices in each of the three parts, they induce at most $2\cdot2\cdot2 = 8$ edges. If there are exactly three vertices in one of the parts, then there is a part with two vertices and a part with one vertex,  i.e., there are at most $3\cdot2\cdot1 = 6$ edges between the parts, and at most one additional edge inside the first part, giving at most $7$ edges. If there are four vertices in one part, then there are at most $4\cdot1\cdot1$ edges between the parts, and at most $\binom 43 = 4$ additional edges inside the first part, i.e., at most $8$ edges in total.\\

\textbf{Case 5:} $X$ contains vertices from only one or two distinct parts.\\ Then there are no edges between these parts, and all possible edges induced by the six vertices are inside the $A_i$'s. Since the construction is iterative, we can use the previous cases to conclude that $X$ induces at most $9$ edges. 
\end{proof}

\end{document}